\documentclass[english,11pt,a4paper]{smfart}

\usepackage[english]{babel}
\usepackage{amsmath}
\usepackage{amsfonts}
\usepackage{amssymb}
\usepackage{amsthm}
\usepackage{a4wide}
\usepackage{mathrsfs}
\usepackage[T1]{fontenc}

\newtheorem{theorem}{Theorem}
\newtheorem{dfn}{Definition}
\newtheorem{lem}{Lemma}
\newtheorem{cor}{Corollary}
\newtheorem{exm}{Example}
\newtheorem{rem}{Remark}

\def\bn{\mathbf{n}}
\def\NN{\mathbb{N}}
\def\C{\mathbb{C}}
\def\Q{\mathbb{Q}}
\def\N{\mathbb{N}}
\def\di{\displaystyle}

\begin{document}
\title[Isochronous centers of polynomial Hamiltonian systems]{Isochronous centers of polynomial Hamiltonian systems and a conjecture of Jarque and Villadelprat}
\author{Jacky Cresson}
\author{Jordy Palafox}

\begin{abstract}
We study the conjecture of Jarque and Villadelprat stating that every center of a planar polynomial Hamiltonian system of even degree is nonisochronous. This conjecture is prove for quadratic and quartic systems. Using the correction of a vector field to characterize isochronicity and explicit computations of this quantity for polynomial vector fields, we are able to describe a very large class of nonisochronous Hamiltonian system of even degree of degree arbitrary large.
\end{abstract}

\maketitle

\setcounter{tocdepth}{3}
\tableofcontents

\newpage

\part{Introduction and main results} 

\section{The Jarque-Villadelprat conjecture}

In this paper, we study centers of planar polynomial Hamiltonian systems in the real case. In particular we focus on isochronous centers. Our main concern is the following conjecture stated by Jarque and Villadelprat in \cite{jarque1}: Let $X$ be a real polynomial Hamiltonian vector field of the form: 
\begin{equation*}
X(x,y)=-\partial_y H(x,y) \partial_ x +\partial_x H(x,y) \partial_y , \ (x,y) \in \mathbb{R}^2
\end{equation*}
where $H(x,y)$ is a real polynomial in the variables $x$ and $y$. The maximum degree of the polynomials $\partial_x H$ and $\partial_y H$ is the degree of the Hamiltonian vector field. \\

\noindent {\bf Conjecture} : {\it Every center of a planar polynomial Hamiltonian system of even degree is nonisochronous.
}\\

The conjecture is known to be true for quadratic systems thanks to a result of Loud in \cite{loud} and in the quartic case by a result of Jarque-Villadelprat in \cite{jarque1}. The proof of Jarque and Villadelprat is based on a careful study of the bifurcations set and seems difficult to extend to an arbitrary degree. The conjecture is open for the other cases despite partial results in this direction obtain by B. Schuman in \cite{sch1,sch2} using an explicit computation of the first coefficients of the Birkhoff normal form and Chen and al. \cite{chen} proving what they call a {\it weak version} of the conjecture, i.e. that any vector fields having only even components is nonisochronous.\\

Different strategies can be used to go further toward this conjecture. A first class of methods can be called {\it geometric} and are related to some special features of Hamiltonian or isochronous centers. We can mention for example the work of L. Gavrilov \cite{gav} and P. Mardesic, C. Rousseau and B. Toni \cite{marde}. Up to now, these methods are unable to reproduce some special results obtained by B. Schuman \cite{sch1,sch2} for classes of polynomial vector fields of arbitrary degree. Another class of methods can be called {\it analytic} and are more or less all dealing with the computations of quantities which be obtained algorithmically like {\it period constants} \cite{fp} and coefficients of normal forms \cite{cs} (see also \cite{f}). However, such methods are usually assumed to be {\it intractable} when one is dealing with a vector field of arbitrary degree (see for example \cite{jarque1} p.337). This is indeed the case when one has no informations on the algebraic structure of these coefficients. Then one is reduced to compute Grobner bases or to use the elimination method. However, one is quickly limited by the computational complexity and the memory size need to perform these computations. Existing results are restricted to polynomials of order $5$. \\

A natural problem is then to look for methods allowing us to bypass these technical limitations. However, one can efficiently obtain more information on the structure of these coefficients. The idea is to separate in these coefficients what is {\it universal} and what is not.

\section{Main results}

In this paper, we bypass this problem using the formalism of {\it moulds} introduced by Jean Ecalle (see \cite{ec2}, \cite{ec3}) and a particular object attached to a vector field called the {\it correction} defined by Ecalle and Vallet in \cite{ev2}. In particular, we obtain a partial answer to the conjecture for arbitrary degree.\\

It is well known that isochronicity of a real center is equivalent to its linearisability (see \cite{chava}, theorem 3.3, p.12). A main property of the correction is that it gives a very useful criterion for linearisability. Indeed, a vector field is linearisable if and only if its correction is zero. As the correction possesses an algorithmic and explicit form which is easily calculable using mould calculus we are able to give more informations on the isochronous set. This strategy was already used by one of us in \cite{cr3}.\\

In the following, we use the classical complex representation of real vector fields (see \cite{llibre}). Let us denote by $X_{lin}=i(x \partial_x-\bar{x}\partial_{\bar{x}})$ and $X_r=P_r(x,\bar{x})\partial_x+\overline{P_r(x,\bar{x})}$
with $x \in \mathbb{C}$, $P_r$ is a homogeneous polynomial of degree $r$, $P_r(x,\bar{x})=\underset{j=0}{\overset{r}{\sum}}p_{r-j-1,j}x^{r-j}\bar{x}^{j}$.

\begin{theorem}
\label{main1}
Let $X$ be a non trivial real Hamiltonian vector field of even degree 2n given by:
\begin{equation*}
X=X_{lin}+\underset{r=2}{\overset{2n}{\sum}}X_r
\end{equation*}
If $X$ satisfies one of the following conditions : \\
$\textit{a})$ there exists $1 \leq r < n-1$ such that $p_{i,i}=0$ for $i=1,...,r-1$ and $Im(p_{r,r})>0$, \\
 $\textit{b})$ $p_{i,i}=0$ for $i=1,...,n-1$,\\ 
then the vector field is nonisochronous.
\end{theorem}

As a consequence we deduce that:
\begin{align*}
X&=X_{lin}+X_2 , \\
X&=X_{lin}+X_2+X_3+X_4 \ \text{with} \ p_{1,1} \geq 0, \\
X&=X_{lin}+X_2+X_3+X_4+X_5+X_6 \ \text{with} \ p_{1,1}>0 \ \text{or} \ p_{1,1}=0 \ \text{and} \ p_{2,2}>0,\end{align*}  are nonisochronous.\\

As a corollary, we obtain the {\it weak version} of the Jarque-Villadelprat conjecture proved by X. Chen and al. \cite{chen}:

\begin{cor}[weak Jarque-Villadelprat conjecture]
Let $X$ be a non trivial real Hamiltonian vector field of even degree 2n given by
$$X=X_{lin} +X_2 +X_4 +\dots +X_{2n},$$
then $X$ is nonisochrnous.
\end{cor}

The proof follows easily from Theorem \ref{main1} as for all $i=1,\dots ,n-1$, we have $p_{i,i}=0$ due to the fact that there exists no odd components. 

\begin{theorem}
\label{main2}
Let $X$ be a non trivial real Hamiltonian vector field of the form:
\begin{equation*}
X=X_{lin}+X_k+...+X_{2l},
\end{equation*}
for $k \geq 2 $ and $l \leq k-1$. Then $X$ is nonisochronous.
\end{theorem}

Using this last theorem, \textit{without any conditions} we have that:
\begin{align*}
X&=X_{lin}+X_2, \\
X&=X_{lin}+X_3+X_4, \\
X&=X_{lin}+X_4+X_5+X_6,
\end{align*} or more funny 
\begin{center}
$X=X_{lin}+\underset{i=47}{\overset{92}{\sum}}X_i$
\end{center} 
are nonisochronous. We see that Theorem 1 and Theorem 2 are complementary to each other. \\

Mixing the proofs of Theorem 1 and Theorem 2 we obtain:  

\begin{theorem}
\label{thm4}
Let $X$ be a non trivial real polynomial Hamiltonian vector field on the form: 
\begin{align*}
X=X_{lin}+X_k+...+X_{2l}+\underset{n=1}{\overset{m}{\sum}}\underset{c_n}{\overset{2(c_n-1)}{\sum}}X_{c_n}
\end{align*}
where $k \geq 2$, $l \leq k-1$ and the sequence $c_n$ is defined by : $c_1=4l$ and $\forall n \geq 2$, $c_n=4(c_{n-1}-1)$, $X$ is nonisochronous.
\end{theorem}

A first example of nonisochronous vector field given by the last theorem is: 
\begin{align*}
X=X_{lin}+X_2+X_4+X_5+X_6.
\end{align*}

\begin{theorem}
\label{thm3}
Let $k \geq 2$ and $l \leq k-1$, a real polynomial Hamiltonian vector field denoted by $X$ on one of these two forms:
\begin{align*}
i) X=X_{lin}+X_k+...+X_{2l}+X_{2l+1}+ \underset{m=r}{\overset{r+n}{\sum}} X_m
\end{align*} 
where $r \geq 2l+2$ and $Im(p_{l,l})>0$ or 
\begin{align*}
ii) X=X_{lin}X_k+...+X_{2l}+X_{4l-1}+\underset{m=r}{\overset{r+n}{\sum}} X_r
\end{align*}
where $X_{2l}$ is nontrivial, $r \geq 4l$, with $Im(p_{2l-1,2l-1})>0$, are nonisochronous. 
\end{theorem}

Using Theorem \ref{thm4} and Theorem \ref{thm3} we easily deduce the classical result that homogeneous perturbations of a linear center are non isochronous (see \cite{sch1}) as well its generalization (see \cite{sch2,cr3}).

\section{Plan of the paper}

In Part \ref{correction}, we give following J. Ecalle and B. Vallet \cite{ev2} the definition of the correction of a vector field and remind some of its properties. We then look more specifically to the correction of polynomial real vector fields. We derive explicit formula allowing us to analyse its structure. \\

In Part \ref{geometry}, we prove that the set of isochronous Hamiltonian centers is an affine variety which can be explicitly described. We also prove that this variety is invariant under a non trivial torus action. \\

Part \ref{proof}, we give the proofs of our main results and some technical Lemmas. \\

We then discuss several perspective for this work.

\newpage

\part{Correction of vector fields and Hamiltonian systems}
\label{correction}
\setcounter{section}{0}

\section{Correction of a vector field}

In this section, we remind the definition of the correction of a vector field following the work of J. Ecalle and B. Vallet \cite{ev2}. In particular, we give the mould expansion of the correction, which plays a central role in our approach to study the linearisability. It must be noted that all these computations can be made in arbitrary dimension.

\subsection{The correction of a vector field}

We denote by $X$ an analytic vector field on $\mathbb{C}^\nu$ at $0$:
\begin{align*}
X=\underset{1 \leq j \leq \nu}{\sum}X_j(x)\partial_{x_j}
\end{align*}
with $X_j(0)=0$ and $X_j(x) \in \mathbb{C}\lbrace x \rbrace$. We can write the vector field $X$ in its \textit{prepared form}:

\begin{dfn}
A vector field $X$ is said $\textit{in prepared form}$ if it is given by
\begin{equation*}
X=X_{lin}+\underset{n \in A(X)}{\sum}B_n
\end{equation*}
where $X_{lin}$ is the linear part of $X$ on the form $X_{lin}=\underset{j}{\sum}\lambda_j x_j \partial_{x_j}$ $B_n$ are homogeneous differential operator of degree $n$ in a given set $A(X)$ which is completely defined by $X$.
\end{dfn}

From the point of view of Analysis, $\textit{homogeneous differential operators}$ are more tractable. An operator $B_n$ is said to be homogeneous of degree $n=(n_1,n_2)$ if for all monomial $x^ly^k$ we have $B_n(x^ly^k)=\beta_n^{l,k}\cdot x^{n_1+l}y^{n_2+k}$ with $\beta_n^{l,k} \in \mathbb{C}$.
\\
\\
In \cite{ev2}, J.Ecalle and B.Vallet introduce the $\it correction$ of a vector field following previous work of G.Gallavotti \cite{gal} and H.Eliasson \cite{el} in the Hamiltonian case. \\

Let us consider a vector field in prepared form. The correction is defined as follows(\cite{ev2}, p.258):

\begin{dfn}
Let $X$ and $Y$ two vectors fields with the same linear part, we assume that $Y$ is linearizable. We denote $A \sim B$ if the vectors fields $A$ and $B$ are formally conjugate. Find a local vector field $Z$ such that:
\begin{align*}
X-Z & \sim Y, \\
[Y,Z]&=0.
\end{align*} 
The correction is the solution $Z$ of this problem.
\end{dfn}

In \cite{ev2}, Ecalle and Vallet prove that the correction of a vector field admits a $\textit{mould expansion}$. \\ 
Precisely, let us denote by $A^*(X)$ the set of the words given by the letters in $A(X)$ using by the concatenation morphism $conc$ on letters :
\begin{align*}
conc : A(X)^p & \rightarrow A^*(X) \\
(n_1,...,n_p) & \mapsto n_1 \cdot n_2 \cdot ... \cdot n_p.
\end{align*} 
for any integer $p$.\\

In the following, a word is denoted by $n_1\cdot n_2 \cdot ... \cdot n_p$ or $n_1n_2...n_p$.

\begin{rem} 
The length of the word $n_1...n_p$ is $p$. The word of length $0$ is denoted by $\emptyset$.
\end{rem}

\begin{dfn}
The set $A^*(X)$ is composed by all the words of all lengths that is, if $\textbf{n} \in A^*(X)$ there exists an integer $p \geq 0$ such that $\textbf{n}=conc(n_1,...,n_p)$ where $n_j \in A(X)$ for $j=1,...,p$. We denote $A^p(X)$ the set of words of length $p$.
\end{dfn}

\noindent For all ${\textbf{n}}=n_1\cdot ...\cdot n_r \in A^*(X)$, we denote:
\begin{align*}
B_{{\textbf{n}}}=B_{n_1} \circ ... \circ B_{n_r}.
\end{align*}

\noindent The correction can be written as (\cite{ev2}, Lemma 3.2 p.267):
\begin{align*}
Carr(X)=\underset{{\textbf{n}} \in A^*(X)}{\sum}Carr^{{\textbf{n}}}B_{{\textbf{n}}},
\end{align*}
or simply $Carr(X)=\underset{\bullet}{\sum}Carr^{\bullet}B_{\bullet}$ following Ecalle's notations. \\
The main point is that the mould $Carr^{\bullet}$ can be computed \textit{algorithmically} using a recursive formula on the length of words.
Precisely for all $n \in A(X)$, let us denote by $\omega (n)$ the quantity: 
\begin{align*}
\omega(n)= \langle n, \lambda \rangle ,
\end{align*}
where the $\langle . , . \rangle$ is the usual scalar product on $\mathbb{C}^n$ and $\lambda$ is the eigensystem of $X_{lin}$. We can extend $\omega$ to a morphism from $(A^*(X), conc)$ to $(\mathbb{C},+)$. The quantity $\omega(n)$ is the weight of the letter $n$.

We have the following theorem (formula 3.42 in \cite{ev2}): 

\begin{theorem}[Variance formula]
The mould of the correction is given by the formula for any word \\ ${\textbf{n}}=n_1\cdot ... \cdot n_r$: 
\begin{align*}
\omega(n_1) Carr^{n_1,n_2,...,n_r}+Carr^{n_1+n_2,n_3,...n_r}=\underset{n_1\textbf{b}\textbf{c}={\textbf{n}}}{\sum}Carr^{n_1 \textbf{c}}Carr^{\textbf{b}}.
\end{align*}
\end{theorem}

The proof of this theorem is nontrivial, it follows from $\textit{variance formula}$ for a vector field discussed in (\cite{ev2}, Prop 3.1 p.270). The variance of a vector field gives many different way to compute the mould of the correction. \\

The main consequence of the previous Theorem is the {\it universal} character of the mould of the correction. Precisely, following the definition of universality used in \cite{cr2}, we have : 

\begin{theorem}[Universality of the Correction's mould] 
There exists a one parameter family of complex functions $C_r : \mathcal{D}_r \subset \C^r \rightarrow \C$, $r\in \N$ such that for all $X$ the correction's mould $Carr^{\bullet}$ defined on $A(X)^*$ is given for all  $\bn \in A(X)^*$ such that $l(\bn )=r$, $r\in \N$ and $\omega (\bn )=0$ by 
\begin{equation}
Carr^{\bn } = C_r (\omega(n_1 ) ,\dots ,\omega (n_r)) .
\end{equation}
\end{theorem}

This property is fundamental concerning our problem as the computation of these coefficients is done once  and for all and does not depends on the value of the coefficients entering the polynomials but on the alphabet generated by the vector field. Up to our knowledge only the mould formalism is able to produce such kind of coefficients allowing to write the correction (this is not the case for example dealing with the classical Lie framework).\\

In the following, we give explicit expressions for $C_1$, $C_2$ and $C_3$.

\subsection{The mould of the correction}

The following theorem concerns precisely the length 1,2 and 3 :

\begin{dfn}
The universal correction functions $C_r:\C^r \rightarrow \C$, $r=1,2,3$ are defined by 
\begin{equation}
C_1 (x)=\left \{ 
\begin{array}{lll}
1 & \ \ & \mbox{\rm if}\ z_1 =0 ,\\
0 & & \mbox{\rm otherwise}.
\end{array}
\right .
\ \ C_2 (z_1 ,z_2 )=
\left \{ 
\begin{array}{lll}
-\di\frac{1}{z_1} & \ \ & \mbox{\rm if}\ z_1+z_2 =0 , z_1\not=0 ,\\
0 & & \mbox{\rm otherwise}.
\end{array}
\right .
\end{equation}
\begin{equation}
C_3 (z_1 ,z_2 ,z_3 )=
\left \{ 
\begin{array}{lll}
\di\frac{1}{z_1 (z_1 +z_2 )} ,& \ \ & \mbox{\rm if}\ z_1+z_2 +z_3 =0,\ z_1\not=0,\ z_1 +z_2 \not= 0, \\
0 & & \mbox{\rm otherwise}.
\end{array}
\right .
\end{equation} 
\end{dfn}

The proof is based on explicit computations which are summarized by the following Lemmas whose proof are given in Appendix. \\
We can remark the values of the correction's mould depend on the weight of the letters. Moreover, 

\begin{dfn} A word ${\textbf{n}} \in A^*(X)$ is said to be resonant if $\omega ( {\textbf{n}})=0.$
\end{dfn}

\begin{lem}
The mould $Carr^\bullet$ verifies : 
\\
$1)$ $Carr^\emptyset$=0, \\
$2)$If $\textbf{n}$ is non resonant, $Carr^{{\textbf{n}}}=0$, \\
$3)$If $\textbf{n}=n_1...n_r$ is such that there exists $j$ satisfying $\omega(n_j)=0$ then $Carr^{\textbf{n}}=0$.
\end{lem}

Moreover, we have:

\begin{lem}
$1)$ If $\omega(n)=0$, $Carr^n=1$, \\
$2)$ If $\omega(n_1\cdot n_2)=0$ with $\omega(n_1)=-\omega(n_2)\neq 0$, we have $Carr^{n_1\cdot n_2}=-\frac{1}{\omega(n_1)}$, \\
$3)$ If $\omega(n_1\cdot n_2 \cdot n_3)=0$ with $\omega(n_j)=0$ $j=1,2,3$, we have $Carr^{n_1\cdot n_2 \cdot n_3}=\frac{1}{\omega(n_1)(\omega(n_1)+ \omega(n_2))}$.
\end{lem}

\subsection{Some computations of the correction mould}

Let us consider the quadratic case, i.e. 
\begin{equation}
X=X_{lin} +X_2 ,
\end{equation}
where $X_{lin}$ is diagonal with eigenvalues $(i,-i)$. \\

The alphabet generated by $X_2$ is given by 
\begin{equation}
A(X)=\left \{ 
n_1=(1,0),\ n_{-1}=(0,1),\ n_3=(2,-1),\ n_{-3}=(-1,2) 
\right \}
.
\end{equation}

All the letters in $A(X)$ are non resonant so that the correction mould is always zero in length $1$. In length $2$ however, some resonant combinations are possible. We have 

\begin{center}
\begin{tabular}{|l|p{2.1cm}|c|}
\hline
 Word \textbf{n} & $Carr^{\textbf{n}}$ \\
\hline
$n_1\cdot n_{-1}$ & $i $ \\
$n_{-1}\cdot n_1$ & $-i$ \\
$n_3 \cdot n_{-3}$ & $\frac{i}{3} $ \\
$n_{-3}\cdot n_3$ & $\frac{-i}{3}$\\
\hline 
\end{tabular}
\end{center}

In length $4$, the correction mould is given by :

\begin{center}
\begin{tabular}{|l|p{2.1cm}|c|}
\hline
 Word \textbf{n} & $Carr^{\textbf{n}}$ \\
\hline
$n_{-3}\cdot n_{-3} \cdot n_{3} \cdot n_{3}$ & $\frac{-i}{54} $ \\
$n_{-3}\cdot n_{-1} \cdot n_{1} \cdot n_{3}$ & $\frac{-i}{12}$ \\
$n_{-3}\cdot n_{-1} \cdot n_{3} \cdot n_{1}$ & $ \frac{i}{12}$ \\
$n_{-3}\cdot n_{1} \cdot n_{-1} \cdot n_{3}$ & $ \frac{i}{6}$ \\
$n_{-3}\cdot n_{1} \cdot n_{1} \cdot n_{1}$ & $\frac{i}{6} $ \\
$n_{-3}\cdot n_{1} \cdot n_{3} \cdot n_{-1}$ & $ \frac{-i}{6}$ \\
$n_{-3}\cdot n_{3} \cdot n_{-3} \cdot n_{3}$ & $ \frac{i}{27}$ \\
$n_{-3}\cdot n_{3} \cdot n_{-1} \cdot n_{1}$ & $0 $ \\
$n_{-3}\cdot n_{3} \cdot n_{1} \cdot n_{-1}$ & $0$ \\
$n_{-3}\cdot n_{3} \cdot n_{3} \cdot n_{-3}$ & $0$ \\
$n_{-1}\cdot n_{-3} \cdot n_{1} \cdot n_{3}$ & $\frac{i}{12} $ \\
$n_{-1}\cdot n_{-3} \cdot n_{3} \cdot n_{1}$ & $ \frac{-i}{12}$ \\
$n_{-1}\cdot n_{-1} \cdot n_{-1} \cdot n_{3}$ & $\frac{i}{6} $ \\
$n_{-1}\cdot n_{-1} \cdot n_{1} \cdot n_{1}$ & $\frac{-i}{2} $ \\
$n_{-1}\cdot n_{-1} \cdot n_{3} \cdot n_{-1}$ & $\frac{-i}{2} $ \\
$n_{-1}\cdot n_{1} \cdot n_{-3} \cdot n_{3}$ & $0$ \\
$n_{-1}\cdot n_{1} \cdot n_{-1} \cdot n_{1}$ & $i$ \\
$n_{-1}\cdot n_{1} \cdot n_{1} \cdot n_{-1}$ & $0$ \\
$n_{-1}\cdot n_{1} \cdot n_{3} \cdot n_{-3}$ & $0$ \\
$n_{-1}\cdot n_{3} \cdot n_{-3} \cdot n_{1}$ & $\frac{-i}{6} $ \\
$n_{-1}\cdot n_{3} \cdot n_{-1} \cdot n_{-1}$ & $\frac{i}{2} $ \\
$n_{-1}\cdot n_{3} \cdot n_{1} \cdot n_{-3}$ & $ \frac{i}{6}$ \\
\hline
\end{tabular}
\begin{tabular}{|l|p{2.1cm}|c|}
\hline
 Word \textbf{n} & $Carr^{\textbf{n}}$ \\
\hline
$n_{1}\cdot n_{-3} \cdot n_{-1} \cdot n_{3}$ & $ \frac{-i}{6}$ \\
$n_{1}\cdot n_{-3} \cdot n_{1} \cdot n_{1}$ & $\frac{-i}{2} $ \\
$n_{1}\cdot n_{-3} \cdot n_{3} \cdot n_{-1}$ & $ \frac{i}{6}$ \\
$n_{1}\cdot n_{-1} \cdot n_{-3} \cdot n_{3}$ & $0$ \\
$n_{1}\cdot n_{-1} \cdot n_{-1} \cdot n_{1}$ & $0$ \\
$n_{1}\cdot n_{-1} \cdot n_{1} \cdot n_{-1}$ & $-i$ \\
$n_{1}\cdot n_{-1} \cdot n_{3} \cdot n_{-3}$ & $0 $ \\
$n_{1}\cdot n_{1} \cdot n_{-3} \cdot n_{1}$ & $\frac{i}{2} $ \\
$n_{1}\cdot n_{1} \cdot n_{-1} \cdot n_{-1}$ & $\frac{i}{2} $ \\
$n_{1}\cdot n_{1} \cdot n_{1} \cdot n_{-3}$ & $\frac{-i}{6} $ \\
$n_{1}\cdot n_{3} \cdot n_{-3} \cdot n_{-1}$ & $\frac{i}{12} $ \\
$n_{1}\cdot n_{3} \cdot n_{-1} \cdot n_{-3}$ & $\frac{-i}{12} $ \\
$n_{3}\cdot n_{-3} \cdot n_{-3} \cdot n_{3}$ & $ 0$ \\
$n_{3}\cdot n_{-3} \cdot n_{-1} \cdot n_{1}$ & $0$ \\
$n_{3}\cdot n_{-3} \cdot n_{1} \cdot n_{-1}$ & $0$ \\
$n_{3}\cdot n_{-3} \cdot n_{3} \cdot n_{-3}$ & $\frac{-i}{27} $ \\
$n_{3}\cdot n_{-1} \cdot n_{-3} \cdot n_{1}$ & $\frac{i}{6} $ \\
$n_{3}\cdot n_{-1} \cdot n_{-1} \cdot n_{-1}$ & $\frac{-i}{6} $ \\
$n_{3}\cdot n_{-1} \cdot n_{1} \cdot n_{-3}$ & $\frac{-i}{6} $ \\
$n_{3}\cdot n_{1} \cdot n_{-3} \cdot n_{-1}$ & $\frac{-i}{12}$ \\
$n_{3}\cdot n_{1} \cdot n_{-1} \cdot n_{-3}$ & $\frac{i}{12} $ \\
$n_{3}\cdot n_{3} \cdot n_{-3} \cdot n_{-3}$ & $\frac{i}{54} $ \\
\hline
\end{tabular} 
\end{center}

\section{Correction of a polynomial vector field}

\subsection{Prepared form and alphabet}

Let $X$ a polynomial vector field in $\mathbb{C}^2$ of the form:
\begin{equation*}
X=X_{lin}+ P_r(x,y)\partial_x+Q_r(x,y)\partial_y
\end{equation*}

where $P_r$ and $Q_r$ are homogeneous polynomials of degree $r$ such that: 
\begin{equation*}
P_r(x,y)=\underset{k=0}{\overset{r}{\sum}}p_{k-1,r-k}x^k y^{r-k}, \\
Q_r(x,y)=\underset{l=0}{\overset{r}{\sum}}q_{r-k,k-1}x^{r-k} y^k.
\end{equation*}
 
In the following, we describe explicitly the prepared form of $X$, the set $A(X)$ and the operators $B_n$ for a given vector field $X$ of the form $X=X_{lin}+\underset{r=2}{\overset{l}{\sum}}X_r$.
 
So we can write :
\begin{align*}
P_r(x,y)\partial_x+Q_r(x,y)\partial_y 
& = \underset{k=1}{\overset{r}{\sum}}\left( p_{k-1,r-k}x^{k-1}y^{r-k}x\partial_x+q_{r-k,k-1}x^{r-k}y^{k-1}y\partial_y \right)+ p_{-1,r}y^r\partial_x + q_{r,-1}x^r\partial_y, \\
&=\underset{k=1}{\overset{r}{\sum}} (\mathit{O}_{(k-1,r-k)}+\tilde{\mathit{O}}_{(r-k,k-1)})+\mathit{O}_{(-1,r)}+\mathit{O}_{(r,-1)},
\end{align*} 
with \begin{align*}
\mathit{O}_{(k-1,r-k)} &=p_{k-1,r-k}x^{k-1}y^{r-k}x \partial_x, \\
\tilde{\mathit{O}}_{(r-k,k-1)} &=q_{r-k,k-1}x^{r-k}y^{k-1} y \partial_y, \\
\mathit{O}_{(-1,r)} &= p_{-1,r}y^r\partial_x, \\
\mathit{O}_{(r,-1)} &= q_{r,-1}x^r\partial_y.
\end{align*}

We want to know if there exist some operators of the same degree among the operators $O_{(k-1,r-k)}$ and $\tilde{O}_{(r-k,k-1)}$. If it is the case, we are allow to gather them in a same operator of the form $B_{(n^1,n^2)}=x^{n^1}y^{n^2}(p_{(n^1,n^2)}x \partial_x +q_{(n^1,n^2)}y\partial_y)$. \\
For that, we have to solve:
\begin{equation*}
k-1=r-\tilde{k}
\end{equation*}
where $k,\tilde{k} \in \lbrace 1 ,...,r \rbrace $. As $\tilde{k} \in \lbrace 1,...,r \rbrace$, $r-\tilde{k} \in \lbrace 0,...,r-1 \rbrace$ hence $r-\tilde{k}+1 \in \lbrace 1,...,r \rbrace$. So there always exist solutions and we have the following lemma: 

\begin{lem} 
\label{letterXr}
For all $r \geq 2$, we can associate an alphabet to the vector field $X_r$, written $A(X_r)$  given by: 
\begin{equation*}
A(X_r)=\lbrace (r,-1), \ (-1,r) , \ (k-1,r-k) \ \text{with} \ k=1,...,r \rbrace .
\end{equation*}
Moreover, we get from $A(X_r)$ the set denoted $\mathcal{B}(X_r)$ of the homogeneous differential operators given by the decomposition : 
\begin{align*}
&B_{(-1,r)}  =p_{-1,r}y^r\partial_x, \\
&B_{(r,-1)}  =q_{r,-1}x^r \partial_y, \\
&B_{(k-1,r-k)}  =x^{k-1}y^{r-k}(p_{k-1,r-k}x\partial_x+q_{k-1,r-k}y\partial_y )
\end{align*}
with $k \in \lbrace 1,...,r \rbrace$.
\end{lem} 

\begin{exm}
We consider the vector field $X=X_{lin}+X_2+X_3$ where 
\begin{align*}
X_2&=\left( p_{1,0}x^2+p_{0,1}xy + p_{-1,2}y^2 \right)\partial_x +\left( q_{-1,2}x^2+q_{1,0}xy+q_{0,1}y^^2\right)\partial_y, \\
X_3&=\left( p_{2,0}x^3+p_{1,1}x^2y+p_{0,2}xy^2+p_{-1,3}y^3 \right) \partial_x+\left( q_{3,-1}x^3+q_{2,0}x^2y+q_{1,1}xy^2+q_{0,2}y^3 \right)\partial_y .
\end{align*}
Hence we obtain the three following alphabets : 
\begin{align*} \
A(X_2) & = \lbrace (2,-1), \ (1,0), \ (0,1), \ (-1,2) \rbrace , \\
A(X_3) & = \lbrace (3,-1), \ (2,0), \ (1,1), \ (0,2), \ (-1,3) \rbrace .
\end{align*}
For example, we also have the element of the set $\mathcal{B}(X_3)$ :
\begin{align*}
&B_{(3,-1)}  = q_{3,-1}x^3 \partial_y, \\
&B_{(2,0)}   =x^2 \left( p_{2,0} x \partial_x + q_{2,0} y \partial_y \right), \\
&B_{(1,1)}  = xy \left( x p_{1,1} x \partial_x + q_{1,1} y \partial_y \right), \\
&B_{(0,2)}  = y^2 \left( p_{0,2} x \partial_x + q_{0,2} y \partial_y \right), \\
&B_{(-1,3)}  = p_{-1,3}y^3 \partial_x .
\end{align*}
\end{exm}

\begin{dfn}
We define the degree of a vector fields as the maximum of the degree of its defining polynomial.\\
In a same way, we define the degree of the homogeneous differential operators $B_n$ or of a Lie bracket of $B_n$ which appear in the decomposition of $X$. We denote by $deg(B_n)$ (resp. $deg([B_n])$) the degree of $B_n$ (resp. $[B_n]$).
\end{dfn}

\begin{lem}
Let $X$ be a vector fields of the form $X=X_{lin}+\underset{r=2}{\overset{m}{\sum}}X_r$ then $X$ admits the alphabet $A(X)=\underset{r=2}{\overset{m}{\cup}}A(X_r)$ and $\mathcal{B}(X)$, the set of homogeneous differential operators of $X$, is given by $\mathcal{B}(X)=\underset{r=2}{\overset{m}{\cup}}\mathcal{B}(X_r)$.
\end{lem}

\begin{proof}
For all $n=(n^1,n^2) \in A(X)$, we define the application :
\begin{align*}
p : A(X) & \rightarrow \mathbb{N} \\
n=(n^1,n^2) & \mapsto n^1+n^2 .
\end{align*}
For every $r \geq 2$, for all $n \in A(X_r)$, we have $p(n)=r-1$, so  $\forall r, r'$ , such that $r \neq r'$, we have $A(X_r) \cap A(X_{r'})=\emptyset$ because $p(A(X_r))\neq p(A(X_{r'}))$. 
Moreover, as $Der(\mathbb{C}^2)=\underset{r \geq 1 }{\bigoplus} Der_r( \mathbb{C}^2 )$ and $\mathcal{B}(X_r) \subset Der_r( \mathbb{C}^2 )$ , then $\mathcal{B}(X_r)\cap \mathcal{B}(X_{r'})=\emptyset$ if $r \neq r'$. 
\end{proof}

The elements of the alphabet $A(X)$ are named $\it{letter}$.

\subsection{Depth}

As we have a one-to-one correspondence between $(A^*(X),conc)$ and $(\mathcal{B}^*(X),\circ)$:
\begin{align*}
A^*(X) & \rightarrow \mathcal{B}^*(X),\\
\textbf{n}=n_1\cdot ... \cdot n_r & \mapsto B_{\textbf{n}}=B_{n_1}\circ ... \circ B_{n_r},\end{align*}
the degree of $[B_{\textbf{n}}]$ gives a natural notion of \textit{depth} for the words defined by:
\begin{dfn}We denote by $p:A^*(X)\rightarrow\mathbb{N}$ the mapping defined by:
\begin{align*}
p(\textbf{n})=deg([B_{\textbf{n}}])-1.
\end{align*}
\end{dfn}
 
\begin{lem}
The mapping $p$ is a morphism from $(A^*(X),conc)$ in $(\mathbb{N},+).$\end{lem}

\begin{proof}
We prove it by induction on the length of the words. Let $n_1,n_2 \in A(X)$,
\begin{align*}
p(n_1\cdot n_2)&=deg([B_{n_1}B_{n_2}])-1 \\
&=deg(B_{n_1})+deg(B_{n_2})-1-1 \\
&=p(n_1)+p(n_2).
\end{align*}
Let $n_1\in A(X)$ and $\textbf{n} \in A^*(X)$, so: 
\begin{align*}
p(n_1\cdot \textbf{n})&=deg([B_{n_1},B_{\textbf{n}}])-1 \\
&=deg(B_{n_1})+deg(B_{\textbf{n}})-1-1 \\
&=p(n_1)+p(\textbf{n}).
\end{align*}
\end{proof}

Let $\mathcal{M}(X)$ be a mould series: 
\begin{align*}
\mathcal{M}(X)=\underset{\textbf{n} \in A^*(X)}{\sum}M^{\textbf{n}}B_{\textbf{n}},
\end{align*}
where $M^\bullet$ is alternal (\cite{ev2}), i.e. $\mathcal{M}(X)$ is \textit{primitive}(\cite{jps}, p.17). In this case, using the {\it projection theorem} (\cite{jps}, p.28), the mould $\mathcal{M}(X)$ can be expressed in the following form :
\begin{align*}
\mathcal{M}(X)=\underset{r \geq 1}{\sum}\frac{1}{r}\underset{l(\textbf{n})=r}{\underset{\textbf{n}\in A^*(X)}{\sum}}M^{\textbf{n}}[B_{\textbf{n}}],
\end{align*}
where $[B_{\textbf{n}}]=[B_{n_1...n_r}]=[...[[B_{n_1},B_{n_2}],B_{n_3}],...],B_{n_{r-1}}],B_{n_r}]$. \\
We have to reorganise this sum using the depth as follows:
\begin{align*}
\mathcal{M}(X)=\underset{d \geq 1}{\sum} \mathcal{M}_d(X),
\end{align*}
where $\mathcal{M}_d(X)=\underset{p(\textbf{n})=d}{\underset{\textbf{n}\in A^*(X)}{\sum}}M^{\textbf{n}}B_{\textbf{n}}$. \\
A useful consequence is that the equation $\mathcal{M}(X)=0$ is equivalent to $\mathcal{M}_d(X)=0$ for all $d\geq 1$.

\subsection{Expression of the correction and criterion of linearisability }

The main property of the correction is that it provides a useful and simple criterion of linearizability. Indeed, we have by definition of the correction (see \cite{ev2}, p.258) :

\begin{lem}
A vector field $X$ is linearizable if and only if $Carr(X)=0$.
\end{lem}

Using the above decomposition, we obtain an explicit criterion for linearizability writing $Carr(X)$ as :
\begin{align*}
 Carr(X)=\underset{p\geq 1}{\sum}\left( \underset{p({\textbf{n}})=p}{\underset{{\textbf{n}}\in A^*(X)}{\sum}} Carr^{{\textbf{n}}}B_{{\textbf{n}}} \right)=\underset{p \geq 1}{\sum}Carr_p(X).
 \end{align*}

\begin{theorem}
\label{caractiso}
A vector fields $X$ is linearizable if and only if $Carr_p(X)=0$ $\forall p\geq 1$.
\end{theorem}

In the following, we derive some properties of the quantities $Carr_p(X)$.

\section{Correction of real polynomial Hamiltonian vector fields}

\subsection{General properties}

An interesting property of the correction is that we just have to consider the even depth, indeed :

\begin{theorem}
Let $X$ be a real Hamiltonian vector fields as above. Its correction in odd depth is zero, i.e.
\begin{equation}
Carr_{2p+1}(X)=0,
\end{equation}
for all integer $p$.
\end{theorem}

This theorem is a consequence of the following lemma : 
 
\begin{lem}
\label{crochet}
For a resonant word \textbf{n}, the related Lie bracket is on the form :
\begin{align*}
[B_{\textbf{n}}]=(xy)^{\frac{p(\textbf{n})}{2}}(P_{\textbf{n}}x \partial_x+Q_{\textbf{n}}y \partial_y).
\end{align*}
\end{lem}

\begin{proof}
For all resonant word $\textbf{n}=n_1\cdot ... \cdot n_r$, the related Lie bracket is :
\begin{align*}
[B_{\textbf{n}}]=x^{\sum n_j^1}y^{\sum n_j^2}(P_{\textbf{n}}x \partial_x+Q_{\textbf{n}}y \partial_y).
\end{align*} 
where $n_j=(n_j^1,n_j^2)$. As \textbf{n} is resonant, we have :
\begin{align*}
\omega (\textbf{n})&=\omega(n_1)+...+\omega(n_r)\\
&=i \left( \sum n_j^1 -  \sum n_j^2 \right)=0,
\end{align*}
then $\sum n_j^1=\sum n_j^2=\alpha \in \mathbb{N}$. We just have to remark that $p(\textbf{n})=p(n_1)+...+p(n_r)=\sum n_j^1+\sum n_j^2=2 \alpha$, then $\alpha = \frac{p(\textbf{n})}{2}$.
\end{proof}

As a consequence, we can restrict our attention to the even components of the correction. For a given integer $p$, terms in $Carr_{2p}$ can be decomposed with respect to the length of words. Precisely, we have 
\begin{equation}
Carr_p (X)=Carr_{p,1} (X)+\dots +Carr_{p,p} (X), 
\end{equation}
where 
\begin{equation}
Carr_{p,j} (X)=\underset{p({\textbf{n}})=p,\, l(\textbf{n})=j}{\underset{{\textbf{n}}\in A^*(X)}{\sum}} Carr^{{\textbf{n}}}B_{{\textbf{n}}} ,
\end{equation}
for $j=1,\dots ,p$.

The main point is that of course, this is a {\it finite} sum. Indeed, as each differential operator entering in the definition are at least of depth one, we can not have more than a word of length $p$ as for all $j\in \NN^*$, $\bn \in A^* (X)$ such that $l(\bn )=j$, we have $p(\bn ) \geq j$.\\

Moreover, some of these terms are easily determined. 

\begin{lem}
\label{propcor}
Let $p\in \NN^*$, we have 
\begin{equation}
\left .
\begin{array}{lll}
Carr_{2p,2p} (X) & = & Carr_{2p,2p} (X_2),\\
Carr_{2p,2} (X) & = & Carr_{2p,2} (X_{p+1} )+\di\sum_{r=2}^p Carr_{2p,2} (X_r ,X_{2p-r+2} ),\\
Carr_{2p,1} (X) & = & Carr_{2p} (X_{2p+1} ) =B_{p,p} .
\end{array}
\right .
\end{equation}
\end{lem}

This Lemma has important implications on the following. In particular, it gives the maximal degree of the homogeneous vector fields $X_r$ entering in the computation of a given correction term. In particular, for $Carr_{2p}$ we have no terms coming from the $X_r$ with $r\geq 2p+2$. \\

The proof of Lemma \ref{propcor} is based on the properties of the set of resonant words with respects to length and depth.\\

The first equality easily follows from the fact that an element of depth $2p$ and length $2p$ is necessarily made of elements of depth $1$ corresponding to operators in $X_2$. \\

The second equality is a direct rewriting of the definition for a length $2$ contribution to the correction term of depth $2p$. We denoted by $\mathcal{W}(X_r)$ the set of weights coming from the component $X_r$ of $X$ and given by :
\begin{equation}
\lbrace \langle n , \lambda \rangle , n \in A(X) \rbrace .
\end{equation}

The first term comes from the following decomposition lemma: 

\begin{lem}
For every $r\geq 2$, $\mathcal{W} (X_r )$ can be decomposed in the following way :
\begin{align*}
\mathcal{W} (X_r )=\mathcal{W}^+(X_r) \cup \mathcal{W}^-(X_r ) \cup \mathcal{W}^0(X_r ),
\end{align*}
where $\mathcal{W}^+(X_r )$ is the set of positive weights coming from $X_r$, $\mathcal{W}^-(X_r)=-\mathcal{W}^+(X_r )$ and $\mathcal{W}_r^0(X)$ is the set of the zero weight.  
\end{lem}

This decomposition shows the interaction between each homogeneous component $X_l$ intervening in $Carr_{2p,2} (X)$. In length $2$, as any weight as its symmetric counterpart, we always have a contribution of $X_{p+1}$. \\

The third one follows directly from a computation:

\begin{lem}
A component $X_r$, $r\geq 2$, produces a resonant letter in $A(X)$ if and only if $r$ is odd. In this case, the letter is unique and given by $n_0=(\frac{r-1}{2}, \frac{r-1}{2})$.
\end{lem}

\begin{proof}
Let $r \geq 2$ be fixed. By Lemma \ref{letterXr} we know the set of letters produced by $X_r$. The two letters $(-1,r)$ and $(r,-1)$ are never resonant. For the other ones given by $(k-1,r-k)$, one must solve the equation of resonance
\begin{equation}
2k-r-1=0,
\end{equation}
for $k=1,...,r$. This equation has a unique solution given by
\begin{equation}
k=\frac{r+1}{2},
\end{equation}
which is valid, as $k$ must be an integer, only when $r$ is odd.
\end{proof}

\subsection{Explicit computation and the fundamental Lemma}

We now explicit the quantities $Carr_p(X)$ when $X$ is a real Hamiltonian polynomial vector field. \\

Let $X$ a polynomial vector field in $\mathbb{C}^2$ of the form:
\begin{equation}
\label{geneform}
X=X_{lin}+ \di\sum_{r=2}^d \left ( P_r(x,y)\partial_x+Q_r(x,y)\partial_y 
\right ) ,
\end{equation}
where $P_r$ and $Q_r$ are homogeneous polynomials of degree $r$. 

\begin{lem}
The complex vector field (\ref{geneform}) corresponds to a real vector field if for all $r=2,\dots ,d$, we have 
\begin{equation}
\label{reality}
\overline{p_{i,j}}=q_{j,i} ,\ i=0,\dots ,r-1,\ j=r-i .
\end{equation}
\end{lem}

The proof follows easily from the fact that $\bar{x}=y$ and $Q_r (x,y)=\overline{P_r(x,y)}$ which gives $Q_r (x,y)=\overline{P_r}(y,x)$ for all $r=2,\dots ,d$.\\
 
Real Hamiltonian systems satisfy moreover the following conditions:

\begin{lem}
\label{hamiltoniancondition}
The complex vector field (\ref{geneform}) corresponds to a real Hamiltonian vector fields if conditions (\ref{reality}) are satisfied and moreover if for all $r=2,\dots ,d$, we have
\begin{equation}
p_{i-1,r-i} =-\di\frac{r-i+1}{i} \overline{p_{r-i,i-1}},\ \ i=1,\dots ,r.
\end{equation}
\end{lem}

We give some examples of relations between the coefficients in $X_2$ and $X_3$ :
\begin{exm}
For the vector field $X_2$ defined above, we have :
\begin{align*}
&p_{1,0}=\frac{-1}{2}\bar{p}_{0,1},\\
&p_{-1,2}=\bar{q}_{2,-1}.
\end{align*}
For the vector field $X_3$, we have :
\begin{align*}
&p_{2,0}=\frac{-1}{3}\bar{p}_{0,2}, \\
&p_{1,1}=-\bar{p}_{1,1}, \\
&p_{-1,3}=\bar{q}_{3,-1}.
\end{align*}
\end{exm}

Under the two previous conditions on the coefficients we have: 

\begin{lem}[Fundamental Lemma]
Let $X$ be a real Hamiltonian vector fields of the form $X=i(x \partial_x - y \partial_{y})+\underset{j=r}{\overset{2r-1}{\sum}}(P_j(x,y) \partial_x+\overline{P_j(x,y)}\partial_{y})$ with $\bar{x}=y$, then : 
\begin{equation*}
Carr_{2(r-1)}(X)=p_{r-1,r-1}+i\left( \underset{k=[\frac{r+1}{2}]+1}{\overset{r}{\sum}} \frac{r(r+1)}{(r-k+1)^2}\vert p_{k-1,r-k} \vert^2 +\frac{r}{r+1} \vert p_{-1,r} \vert ^2 \right).
\end{equation*}
\end{lem}

\begin{proof}
The computation of the correction in depth $2(k-1)$ requires to know which operators appear in it and the alphabet. In this depth, there are just the homogeneous differential operators from the polynomial $P_r$ and  the resonant letter from the polynomial $P_{2r-1}$. \\
The alphabet given by $P_r$ is $A(X_r)=\lbrace (r,-1), (-1,r) , (k-1,r-k) \ with \ k=1,...,r \rbrace .$ \\
We also have the homogeneous differential operators : 
\begin{equation*}
B_{(-1,r)}=p_{-1,r}y^r\partial_x, \\
B_{(r,-1)}=q_{r,-1}x^r \partial_y, \\
B_{(k-1,r-k)}=x^{k-1}y^{r-k}(p_{k-1,r-k}x\partial_x+q_{k-1,r-k}y\partial_y )
\end{equation*}
with $k \in \lbrace 1,...,r \rbrace$.  \\
Here the resonant words are :
\begin{align*}
&((k-1,r-k),(r-k,k-1)), \\
&((r-k,k-1),(k-1,r-k)), \\
&((-1,r),(r-1)), \\
&((r,-1),(-1,r)).
\end{align*}
So the Lie brackets associated, using the real and Hamiltonian conditions, are: 
\begin{align*}
[B_{(k-1,r-k)},B_{(r-k,k-1)}]&=\frac{r(r+1)(2k-r-1)}{(r-k+1)^2}\vert p_{k-1,r-k} \vert ^2(xy)^{r-1}(x \partial_x-y \partial_y)\\
&=-[B_{(r-k,k-1)},B_{(k-1,r-k)}], \\
[B_{(-1,r)},B_{(r,-1)}]&=-r\vert p_{-1,r} \vert^2(xy)^{r-1}( x \partial_x - y \partial_y)=- [B_{(r,-1)},B_{(-1,r)}].
\end{align*}
The mould of the correction in this resonant word are:
\begin{align*}
Carr^{(k-1,r-k)\cdot (r-k,k-1)}&=\frac{i}{2k-r-1}, \\
Carr^{(-1,r)\cdot (r,-1)}&=\frac{-i}{r+1}.
\end{align*}
The only resonant letter in $P_{2k-1}$ is $p_{r-1,r-1}$ with the operator $B_{(r-1,r-1)}=p_{r-1,r-1}(xy)^{r-1}( x \partial_x - y \partial_y).$
Using the alternality of the mould $Carr^\bullet$ and the skew-symmetric of the Lie brackets, we get the formula : 
\begin{equation*}
Carr_{2(k-1)}(X)=p_{r-1,r-1}+i\left( \underset{k=[\frac{r+1}{2}]+1}{\overset{r}{\sum}} \frac{r(r+1)}{(r-k+1)^2}\vert p_{k-1,r-k} \vert^2 +\frac{r}{r+1} \vert p_{-1,r} \vert ^2 \right).
\end{equation*}
\end{proof}

\subsection{Examples of computations}
We will give in this section a few examples of computations of the correction.

\subsubsection{The quadratic case}
We consider the real Hamiltonian vector field :
\begin{align*}
X=X_{lin}+X_2
\end{align*}
 where $X_2=(p_{1,0}x^2+p_{0,1}xy+p{-1,2}y^2)\partial_x+(q_{2,-1}x^2+q_{0,1}xy+q_{1,0})\partial_y$. \\
 Using the decomposition in homogeneous differential operators, we obtain the following operators :
 \begin{align*}
 &B_{(1,0)}=x(p_{1,0}x \partial_x+q_{1,0}y\partial_y),\\
 &B_{(0,1)}=y(p_{0,1}x\partial_x+q_{0,1}y\partial_y), \\
 &B_{(-1,2)}=p_{-1,2}y^2\partial_x, \\
 &B_{(2,-1)}=q_{2,-1}x^2 \partial_y.
 \end{align*}
 We also have the alphabet : $A(X_2)= \lbrace (1,0), (0,1), (2,-1), (-1,2) \rbrace$.
 
 The first none trivial correction is in depth 2, because we don't have any resonant letter.
 
So, the correction in depth 2 is given by:
\begin{align*}
Carr_2(X)&=Carr_{2,2}(X)\\
&=\underset{\omega(\textbf{n})=0}{\underset{p(\textbf{n})=2}{\underset{\textbf{n}\in A^*(X)}{\sum}}}Carr^{\textbf{n}}B_{\textbf{n}} \\
&=\underset{\omega(\textbf{n})=0}{\underset{p(\textbf{n})=2}{\underset{\textbf{n}\in A^*(X)}{\sum}}}\frac{1}{\ell(\textbf{n})} Carr^{\textbf{n}}B_{[\textbf{n}]} \\
&=\frac{1}{2}\left( Carr^{(1,0)\cdot (0,1)} B_{[(1,0)\cdot (0,1)]} + Carr^{(0,1)\cdot (1,0)}B_{[(0,1)\cdot (1,0)]} \right. \\
& \left. +Carr^{(2,-1)\cdot (-1,2)} B_{[(2,-1)\cdot (-1,2)]}+
Carr^{(-1,2)\cdot (2,-1)} B_{[(-1,2)\cdot (2,-1)]} \right).
\end{align*}
 
By the Fundamental Lemma, we finally have : 
 \begin{align*}
 Carr_2(X)=i\left( 6 |p_{1,0}|^2+ \frac{2}{3} |p_{-1,2}|^2\right).
 \end{align*}
 
\subsubsection{The cubic case}
We consider the real Hamiltonian vector field :
\begin{align*}
X=X_{lin}+X_2+X_3
\end{align*}
with $X_2=(p_{1,0}x^2+p_{0,1}xy+p{-1,2}y^2)\partial_x+(q_{2,-1}x^2+q_{0,1}xy+q_{1,0})\partial_y$ and $X_3=(p_{2,0}x^3+p{1,1}x^2y+p_{0,2}xy^2+p_{-1,3}y^3)\partial_x+(q_{3,-1}x^3+q_{2,0}x^2y+q_{1,1}xy^2+q_{0,2}y^3)\partial_y$. \\
As above, the first none trivial correction is in depth 2 :
\begin{align*}
Carr_2(X)&=Carr_{2,1}(X_3)+Carr_{2,2}(X_2).
\end{align*}
The only operator which is of depth 2 from $X_3$ is given by its resonant letter $(1,1)$ and the operator $B_{1,1}=xy(p_{1,1}x \partial_x+ q_{1,1}y \partial_y$.
So, using the previous result on the quadratic case, we have :
\begin{align*}
Carr_2(X)=p_{1,1}+i \left(  6 |p_{1,0}|^2+ \frac{2}{3} |p_{-1,2}|^2\right).
\end{align*}
In depth 4, we have :
\begin{align*}
Carr_4(X)&=Carr_{4,2}(X_3,X_3)+Carr_{4,3}(X_3,X_2,X_2)++Carr_{4,4}(X_2,X_2,X_2,X_2).
\end{align*}
The different $Carr_{i,j}(X)$ are given by :
\begin{align*}
&Carr_{4,2}(X_3,X_3)=i( 12 |p_{2,0}|^2+\frac{3}{4}|p_{-1,3}|^2), \\
&Carr_{4,3}(X_3,X_2,X_2)=-i( 120Im(p_{2,0}\bar{p}_{1,0}^2)+\frac{26}{3}Im(\bar{p}_{-1,3}p_{-1,2}\bar{p}_{1,0})+40Im(p_{2,0}p_{-1,2}p_{1,0}))\\
&Carr_{4,4}(X_2,X_2,X_2,X_2)=i\left(-144|p_{1,0}|^4+12|p_{1,0}|^2|p_{-1,2}|^2-\frac{8}{9}|p_{-1,2}|^4+40Re(p_{-1,2}p_{0,1}^3) \right).
\end{align*}
\subsubsection{The quartic case}
We consider the real Hamiltonian vector field :
\begin{align*}
X=X_{lin}+X_2+X_3+X_4
\end{align*}
with $X_2$ and $X_3$ as above and $X_4=(p_{3,0}x^4+p_{2,1}x^3y+p_{1,2}x^2y^2+p_{0,3}xy^3+p_{-1,4}y^4)\partial_x+(q_{4,-1}x^4+q_{3,0}x^3y+q_{2,1}x^2y^2+q_{1,2}xy^3+q_{0,3}y^4)\partial_y$. \\
The correction in depth 2 is the same as the cubic case. \\
In depth 4, we have :
\begin{align*}
Carr_4(X)&=Carr_{4,2}(X_4,X_2)+Carr_{4,2}(X_3,X_3)\\
& +Carr_{4,3}(X_3,X_2,X_2)+Carr_{4,4}(X_2,X_2,X_2,X_2),
\end{align*}
where $Carr_{4,2}(X_4,X_2)$ is given by : 
\begin{align*}
Carr_{4,2}(X_4,X_2)=i\left(12 Re(p_{2,1}\bar{p}_{1,0})+8Re(p_{3,0}p_{-1,2})\right).
\end{align*}

\subsubsection{Maple program}
For the interested readers, we can send some Maple program to compute the correction of a polynomial vector fields.
\newpage

\part{The isochronous center affine variety}
\label{geometry}
\setcounter{section}{0}

In this part, we prove that the set of isochronous center is a rational affine variety which is invariant under a non trivial $\C^*$ action. This affine variety is moreover explicitly described. We also give estimates on the growth of the degree of each rational polynomials entering in this description as well as the growth of the rational coefficients.

\section{Affine variety of isochronous center}

We consider real vector fields written in complex form as $X=X_{lin}+P(x,y)\partial_x+Q(x,y)\partial_y$ where $P$ and $Q$ are polynomials with coefficients in $\mathbb{C}$ such that $\overline{P(x,y)}=Q(y,x)$. We denote by $N(d)$ the number of independent coefficients defining $P$ and by $\mathbf{p}$ any element of this set. By the reality condition, the coefficient of $Q$ can be deduced from those of $P$. We then identify the set of complex polynomials of a given degree $d$ with $\C^{N(d)}$, where $N(d)$ is given by $N(d)=\di\frac{(d-1)(d+4)}{2}$.\\  

We denote by $\mathscr{L}_d$ the set of polynomial perturbations $(P,Q)$ of degree $d$ such that $X$ is linearizable. The set $\mathscr{L}$ can be seen as a subset of $\C^{N(d)}$. Precisely, we have :

\begin{theorem}[Geometric structure]
\label{algebraic}
For all $d\geq 2$, the set $\mathscr{L}_d$ of isochronous centers is an affine variety over $\mathbb{Q}$ in $\C^{N(d)}$.
\end{theorem}

The proof is based on a precise description of the algebraic form of the correction. For all $\bn \in A^* (X)$, let us denote by $P(\bn )$ and $Q(\bn )$ the coefficients given by Lemma \ref{crochet} and satisfying 
\begin{align*}
[B_{\textbf{n}}]=(xy)^{\frac{p(\textbf{n})}{2}}(P(\textbf{n})x\partial_x+Q(\textbf{n})y\partial_y).
\end{align*}

We have :

\begin{theorem}[Algebraic structure]
\label{rational}
For all $p\in \N^*$ the correction term $Carr_{2p} (X)$ has the form 
\begin{equation}
Carr_{2p} (X) = (xy)^{p} 
\left [ 
Ca_{2p} (\mathbf{p}) x\partial_x +\overline{Ca_{2p} (\mathbf{p})} y \partial y 
\right ] ,
\end{equation}
where 
\begin{equation}
\label{formulacar}
Ca_{2p} (\mathbf{p}) =\di\sum_{i=1}^{2p} \di\frac{1}{i!} Ca_{2p,i} (\mathbf{p}) ,
\end{equation}
with
\begin{equation}
Ca_{2p,i} (\mathbf{p}):= \underset{p(\bn )=2p,\, l(\bn )=i}{\underset{\bn\in A^*(X)}{\sum}} Carr^{\bn} P(\bn) ,\ \ i=1,\dots ,2p.
\end{equation}

The quantities $Ca_{2p,i} (\mathbf{p})$, $i=1,\dots ,2p$, are explicit polynomials of degree $i$ in the coefficients of $P$ with coefficients in $\Q$ if $i$ is even and $i\Q$ otherwise. Moreover, these polynomials can be computed algorithmically using recursive formula. 
\end{theorem}

The proof of this theorem is a consequence of two results. First, nested Lie brackets have a very special shapes which can be easily computed. Precisely, we have :

\begin{lem} 
\label{polycro}
For all $\bn \in A^*(X)$, the coefficients $P(\bn )$ and $Q(\bn )$ are polynomials in $\mathbb{Z}[\C^{N(d)}]$ of degree $l(\bn )$ and defined recursively on the length of $\bn$ by 
\begin{equation}
\left .
\begin{array}{ll}
P (n\bn )&=(\mid \bn \mid^1 -n^1) p_n P(\bn ) +\mid \bn \mid^2 q_n P( \bn )- n^2 p_n Q(\bn ), \\ 
Q (n\bn )&=(\mid \bn \mid^2-n^2) q_n Q(\bn ) +\mid \bn \mid ^1 p_n Q(\bn )-n^1 q_n P(\bn ).
\end{array}
\right .
\end{equation}
where for $\textbf{n}=n_1\cdot ... \cdot n_r \in A^*(X)$, we let $|n|^j=n_1^j+...+n^j_r$, $j=1,2$ with $n_i=(n^1_i,n^2_i)$.
\end{lem}

Second, the correction mould can also be computed by a recursive formula from which we deduce :

\begin{lem}
\label{valcorrec}
For all $\bn \in A^* (X)$, the mould $Carr^{\bn}$ belongs to $\Q$ if $l(\bn)$ is odd and $i\Q$ if $l(\bn)$is even.
\end{lem}

The proof of Theorem \ref{rational} easily follows. 

\subsection{Proof of Theorem \ref{algebraic}}

Using Theorem \ref{rational} and the characterization of isochronous centers given by Theorem \ref{caractiso}, the set $\mathcal{L}_d$ is defined by the zero set of an infinite family of polynomials over $\C^{N(d)}$ given by
\begin{equation}
\mathcal{L}_d =\left \{ 
P\in \C^{N(d)} ,\ Ca_{2p} (X) =0,\ p\geq 1 
\right \} .
\end{equation}
We define the {\it ascending chain of ideals} $I_k$ generated by $\langle Ca_2 (\mathbf{p}) ,\dots ,Ca_{2k} (\mathbf{p})\rangle$  in $\C [\mathbf{p} ]$. By the {\it Hilbert Basis Theorem} (see \cite{cox}, Theorem 4, p.77), there exists an $M(d)\in \N^*$ such that $I_M =I_{M+1} =\dots $. We denote by $\mathcal{I}_d$ the resulting ideal. As a consequence, the set $\mathcal{L}_d$ can be obtained as (see \cite{cox}, Definition 8,p.81)
\begin{equation}
\mathcal{L}_d =\mathbf{V}(\mathcal{I}_d )=\left \{ \mathbf{p} \in \C^{N(d)} \mid\ f(\mathbf{p} )=0\ \mbox{\rm for all}\  f\in \mathcal{I}_d \right \} ,
\end{equation}
and corresponds to the affine variety (see \cite{cox}, Proposition 9 p.81) defined by 
\begin{equation}
\mathcal{L}_d =\mathbf{V}(f_1 ,\dots ,f_{s(d)} )=\left \{ \mathbf{p} \in \C^{N(d)} \mid\ f_i (\mathbf{p} )=0\ \mbox{\rm for all}\  i=1,\dots ,s(d) \right \} ,
\end{equation}
where the finite family of polynomials $f_i$, $i=1,\dots ,s(d)$ is a generating set of $\mathcal{I}_d$. As the polynomials defining this variety have coefficients in $\Q$ or $i\Q$ this concludes the proof.

\begin{rem}
Theorem \ref{algebraic} together with Theorem \ref{rational} gives explicit informations on the degree as well as on the growth of the rational coefficients entering in the definition of the affine variety. A natural question is up to which extend these informations can be used to provide a natural upper bound on the number of generators for the ideal generating $\mathcal{L}_d$ thanks to a constructive version of the Hilbert basis theorem. This will be explored in another work.
\end{rem}

\section{$\mathbb{C}^*$-invariance}

The resonant character of the correction has an interesting consequence on the rational algebraic variety of isochronous center. Indeed, let us consider the following action of $\mathbb{C}^*$ :

\begin{dfn}
Let $\lambda \in \mathbb{C}^*$, we denote by $T_\lambda$ the map
\begin{align*}
T_\lambda : \mathscr{C}^{N(d)} &\rightarrow \mathscr{C}^{N(d)} \\
p_{\bullet} & \mapsto \lambda^{\omega(\bullet)}p_{\bullet}
\end{align*}
where $\bullet$ is an arbitrary letter.
\end{dfn}

We extend this action for all monomials $p_\textbf{n}=p_1^{n_1}...p_r^{n_r}$, with \textbf{n} is a word $n_1 \cdot ... \cdot n_r$, we have:
\begin{align*}
T_\lambda(p_\textbf{n})&= T_\lambda(p_1^{n_1}...p_r^{n_r}), \\
&= \lambda^{\omega(n_1)+...+ \omega(n_r)}p_\textbf{n} \\
&=\lambda^{\omega(\textbf{n})}p_{\textbf{n}}.
\end{align*}

We denoted by $[B_\textbf{n}]=[B_{n_1 \cdot ... \cdot n_r}]=(xy)^{\frac{p(\textbf{n})}{2}}(P(\textbf{n)}x  \partial_x+Q(\textbf{n})y \partial_y)$ where, as we have just shown, $P(\textbf{n})$ and $Q(\textbf{n})$ are polynomial in the coefficient of $B_{n_1},...,B_{n_r}$. We have the following lemma which show the $\mathbb{C}^*$-invariance:
\begin{lem}
For all resonant word \textbf{n}, we have :
\begin{align*}
T_\lambda(P(\textbf{n}))=P(\textbf{n}) \ \text{and} \ T_\lambda(Q(\textbf{n}))=Q(\textbf{n}).
\end{align*}
\end{lem}

\begin{proof}
By definition of a resonant word, we have $\omega(\textbf{n})=0$. So :
\begin{align*}
T_\lambda(p_\textbf{n})= \lambda^{\omega(\textbf{n})}p_\textbf{n}= \lambda^0 p_{\textbf{n}}=p_{\textbf{n}}.
\end{align*}
\end{proof}

Finally we can generalise this lemma in the following corollary :
\begin{cor}
For all $\lambda \in \mathbb{C}^*$, the algebraic variety $\mathscr{L}_d$ of the isochronous centers is invariant under the action of $T_{\lambda}$. 
\end{cor}

\begin{proof}
To prove this corollary we just have to remind that only the resonant word contribute to the linearisability. We can conclude by the above lemma.
\end{proof}

\newpage
\part{Proof of the main results}
\label{proof}
\setcounter{section}{0}

\section{Proof of Theorem \ref{main1}}

Let $X$ be a real Homaltonian vector field of even degree $2n$ of the form :
\begin{align*}
X=X_{lin}+\underset{r=2}{\overset{2n}{\sum}}X_r.
\end{align*}
For each $X_r$, we can associate its depth as follows :
\begin{center}
\begin{tabular}{|l|p{2.1cm}|c|}
\hline
\bf $X_r$ & Depth \\
\hline
$X_2$ & 1 \\
$X_3$ & 2 \\
$X_3$ & 3 \\
... & ... \\
$X_{2n-1}$ & $2n-2$ \\
$X_{2n}$ & $2n-1$ \\
\hline
\end{tabular}
\end{center} 
By Theorem 7, we are only interested by the even depth. As a consequence, we look for all possible combinations of arbitrary length which give rise under Lie bracket to an even depth vector field. In the following, we denote by $[X_{k_1} ,\dots ,X_{k_r} ]$ the set of operators that one obtain by nested Lie brackets of homogeneous differential operators $B_{n_i}$ coming from $X_{k_i}$, $i=1,\dots ,r$. As an example, we have :
\begin{center}
\begin{tabular}{|l|p{2.1cm}|c|}
\hline
\bf $X_r$ and $[X_r,X_{r'}]$ & Depth \\
\hline
$X_3$ & 2 \\
$[X_2,X_2]$ & 2 \\
\hline
\hline
$X_5$ & 4 \\
$[X_4,X_2]$ & 4 \\
$[X_3,X_3]$ & 4 \\
$[X_3,X_2,X_2]$ & 4 \\
\hline
\hline
$X_7$ & 6 \\
$[X_6,X_2]$ & 6 \\
$[X_5,X_3]$ & 6 \\
$[X_4,X_4]$ & 6 \\
$[X_5,X_2,X_2]$ & 6 \\
$[X_4,X_3,X_2]$ & 6 \\
$[X_4,X_2,X_2,X_2]$ & 6 \\
$[X_3,X_2,X_2,X_2,X_2]$ & 6 \\
$[X_2,X_2,X_2,X_2,X_2,X_2]$ & 6 \\
\hline 
\end{tabular} 
\end{center}

The correction in depth $2$ is given by :
\begin{align*}
Carr_2(X)=Carr_{2,1}(X_3)+Carr_{2,2}(X_2),
\end{align*}
As the depth is a morphism we have the contribution of the Lie bracket of $X_2$, we also have the contribution of the resonant letter of $X_3$. By the Fundamental Lemma, the correction is given by 
\begin{align*}
Carr_2(X)=p_{1,1}+i \left( 6\vert p_{1,0} \vert^2+ \frac{2}{3}\vert p_{-1,2}\vert\right) .
\end{align*}
By the linearisation criterion, we must have $Carr_2(X)=0$. This implies that
\begin{equation}
\left \{
\begin{array}{lll}
\mbox{\rm Re} (p_{1,1}) & = & 0 ,\\
-\mbox{\rm Im} (p_{1,1}) & = & 6\vert p_{1,0} \vert^2+ \frac{2}{3}\vert p_{-1,2}\vert .
\end{array}
\right .
\end{equation}
As $X$ is real and Hamiltonian, the first equation is always satisfied. The second one has only a non trivial solution if  and only if $Im(p_{1,1}) < 0$. The situation when $Im(p_{1,1}) \geq 0$ leads to two distinct cases. When $Im(p_{1,1})=0$, the Birkhoff sphere reduce to $0$ and we obtain $X_2 =0$. When $Im(p_{1,1}) >0$, the equation can not be satisfied and the vector field is then nonisochronous.\\

Assume that $p_{1,1}=0$ then $X_2 =0$ and we are reduce to the case 
\begin{equation}
X=X_{lin} +X_3 +\dots +X_r .
\end{equation}
As $p_{1,1}=0$, the first non zero term of the correction is $Carr_4 (X)$. By the fundamental Lemma, the term $Carr_4 (X)$ has exactly the same algebraic structure than the preceding $Carr_2 (X)=0$ case. The role of $X_2$ is played by $X_3$ and the role of the resonant term of $X_3$ is played by the resonant term of $X_5$. Here again, we recover the same dichotomy between the case $Im(p_{2,2})=0$ and $Im (p_{2,2} )>0$. In the first case, we obtain that $X_3 =0$ and we are leaded to the same situation as before. Otherwise if $Im (p_{2,2} )>0$ the term $Carr_4 (X)$ can not be zero and we have a nonisochronous center.\\

The preceding discussion is representative of the general strategy of proof. Let us assume that $p_{j,j}=0$ for $j=1,...,r-1$. Then, we prove by induction that $X_2 = \dots =X_{r-1}=0$. In order to finish the proof, two cases must be discussed depending the value of $r$.\\ 

{\bf Case 1 : $r<n-1$}. The component $X_{2r+1}$ is non trivial due to the condition $Im(p_{r,r}) >0$. By the  Fundamental Lemma and the linearisability criterion, we must have  
\begin{equation}
Carr_{2r} (X)=p_{r,r}+i\left( \underset{k=[\frac{r+1}{2}]+1}{\overset{r}{\sum}} \frac{r(r+1)}{(r-k+1)^2}\vert p_{k-1,r-k} \vert^2 +\frac{r}{r-1} \vert p_{-1,r} \vert ^2 \right)=0
\end{equation}
As $Im(p_{r,r}) >0$, this equality can not be satisfied and $X$ is nonisochronous.\\

{\bf Case 2 : $r=n-1$}. In this case, we are reduced to an homogeneous perturbation of degree $2n$ and the correction is given by 
\begin{align*}
Carr_{2n-2} (X)=\underset{k=[\frac{r+1}{2}]+1}{\overset{r}{\sum}} \frac{r(r+1)}{(r-k+1)^2}\vert p_{k-1,r-k} \vert^2 +\frac{r}{r-1} \vert p_{-1,r} \vert ^2 =0
\end{align*}
As $X_{2n}$ is nontrivial, this equation can not be satisfied and $X$ is nonisochrnous. \\

This concludes the proof of the Theorem.

\section{Proof of Theorem \ref{main2}}

The proof follows the lines of those of Theorem \ref{main1}. We have to distinguish two cases : $k$ is even or odd.\\

{\bf Case 1: $k$ is even}. The vector field $X_k$ does not contain resonant terms. As a consequence, its first contribution to the correction appears in depth $2(k-1)$ corresponding to resonant Lie brackets of homogeneous differential operators in $X_k$ of length two. As $l\leq k-1$, this implies that $2l-1 <2l \leq 2(k-1)$. The even component between $X_k$ and $X_{2l}$ will come into play in the correction only with a greater depth as $2(k-1)$ by Lie brackets of length at least two. In the same way, for odd components, the resonant term will intervene in the correction with a strictly smaller depth in length one and the other terms in depth greater than $2(k-1)$ by a Lie brackets of length at least two.\\

As a consequence, the correction term coming from an even $k$ is given by : 
\begin{equation}
Carr_{2(k-1)}(X)=i\left( \underset{k=[\frac{r+1}{2}]+1}{\overset{r}{\sum}} \frac{r(r+1)}{(r-k+1)^2}\vert p_{k-1,r-k} \vert^2 +\frac{r}{r-1} \vert p_{-1,r} \vert ^2 \right).
\end{equation}
In order to satisfy the linearisability criterion, we must have $Carr_{2(k-1)}=0$. If the component $X_k$ is non trivial then the system is already non isochronous. Otherwise, we have $X_k =0$ and we are leaded to the same problem but with an odd component.\\

{\bf Case 2 : $k$ is odd}.  In this case the vector field $X_k$ contain a resonant homogeneous operator. Let us write $k=2m+1$ then $B_{m,m}$ is of depth $2m$ and weight zero. We have 
\begin{equation}
Carr_{k-1} (X)=Carr_{2m} (X) =B_{m,m} .
\end{equation}
By the linearisability criterion, $Carr_{k-1} (X)=0$  and the resonant term $B_{m,m}$ in $X_k$ is zero. The contribution of $X_k$ in length $2$ follows the same argument as for the even case and we deduce that finally $X_k =0$. \\

As a consequence, we can prove by induction that in order to be linearisable the components $X_k, \dots ,X_{2l}$ must be zero. But, by assumption, we have that $X_{2l}$ is non trivial. As a consequence, the vector field $X$ is necessarily nonisochronous. 

\section{Proof of Theorem \ref{thm4}}

The strategy of proof follows those of Theorem \ref{main2}. The main observation is that there exists no interactions between each family of vector fields $\{ X_k ,\dots ,X_{2l} \}$ and $\{ X_{c_n} ,\dots ,X_{2(c_n -1)} \}$, $n=1,\dots ,m$. Indeed, let us first analyse the depth of all these objects. We have : 
\begin{center}
\begin{tabular}{|l|p{2.1cm}|c|}
\hline
\bf $X_r$ & Depth \\
\hline
$X_k$ & $k-1$ \\
$X_{k+1}$ & $k$ \\
$\vdots$ & $\vdots$ \\
$X_{2l-1}$ & $2l-2$ \\
$X_{2l}$ & $2l-1$ \\
0 & 0 \\
$X_{4l}$ & $4l-1$ \\
$\vdots$ & $\vdots$ \\
$X_{8l-2}$ & $8l-3$ \\
... & ... \\
$X_{c_m}$ & $c_m-1$ \\
$\vdots$ & $\vdots$ \\
$X_{2(c_m-1)}$ & $2(c_m-1)-1$ \\
\hline
\end{tabular} 
\end{center}
Following the same lines as for Theorem \ref{main2}, we see that the arguments based on the contributions of a given component belonging to $\{ X_k ,\dots ,X_{2l} \}$ are valid. In other words, we easily proved that in order to be linearisable, then one must have $X_k =\dots =X_{2l-2} =0$. The last argument concerning $X_{2l}$ is also satisfies because the first contribution of $X_{2l}$ to the correction is of length two and depth $4l-2$ which is not disturbed by terms of the remaining family $\{ X_{c_n} ,\dots ,X_{2(c_n -1)} \}$, $n=1,\dots ,m$ as the minimal contribution of these terms to the correction is in depth $4l-1$.\\

As a consequence, a vector fields of this type will be linearisable if $X_k =\dots =X_{2l} =0$.\\

By the same argument, we see that there exists no interaction between the family $\{ X_{c_1} ,\dots$ $\dots ,X_{2(c_1 -1)} \}$ and the remaining one $\{ X_{c_n} ,\dots ,X_{2(c_n -1)} \}$, $n=2,\dots ,m$. We deduce that a linearisable vector field of this type must satisfy $X_{c_1} =\dots =X_{2(c_1 -1)}$.\\

By induction, we easily deduce that a vector field of this type is linearisable if and only if all the components vanish. As by assumption we are considering a non trivial vector field, we are leaded to a contradiction and the vector field is necessarily nonisochronous. 

\newpage

\part{Conclusion and perspectives}
\setcounter{section}{0}

\section{Toward a complete proof of the Jarque-Villadelprat conjecture}

Our main results give a very strong support to the Jarque-Villadelprat conjecture. The remaining cases always deal with the role of the resonant term in the deformation of the Birkhoff's spheres. However, the phenomenon which is working for a quartic perturbation, which was precisely studied by Jarque and Villadelprat using geometrical methods, must applies in the same way for arbitrary degrees. Indeed, as we have seem in our derivation of the main results, the algebraic structure of the correction during the cancellation process does not change and can be closely investigated. We have then some directions in order to solve completely the Jarque-Villadelprat conjecture :
\begin{itemize}
\item Can we prove using an extension of our method the quartic case ?

\item Can we extend the geometrical method of Jarque-Villadelprat in the cases which are let open by our work ? 

\item Can we prove the remaining cases using other methods ?
\end{itemize}
We believe that a better understanding of the algebraic structure of the correction will be of importance in order to go further.

\section{Effective Hilbert basis and the isochronous centers affine variety} 

A second aspect of our work is the explicit and algorithmic description of the isochronous center affine variety. As already said, we have informations about the growth of the degree and the coefficients entering in the description of this variety. A natural question is then to look for effective version of the Hilbert basis theorem in order to get some informations about the minimal number of generators of the ideal. The isochronous centers seem to be more tractable than the usual center. However, it is clear that any advance in this direction will have consequences on the local 16th Hilbert problem. Indeed, the same kind of combinatoric and tools can be used to obtain analogous information for centers of polynomial vector fields (see \cite{cs}). 

\section{Isochronicity for complex Hamiltonian systems}

In \cite{llibre}, the authors study isochronicity of complex Hamiltonian systems when the linear part has for spectrum $(1,-1)$. Our method and results extend naturally to this case and give an explicit and algorithmic description of the isochronous centers affine variety. This will be the subject of a forthcoming work.

\newpage
\begin{appendix}

\section{Notations}

\begin{itemize}
\item $A(X)$ alphabet associated to a vector field $X$;
\item $A^*(X)$ set of words given by the alphabet $A(X)$;
\end{itemize}
\begin{itemize}
\item $\mathcal{B}(X)$ set of homogeneous differential operator associated to a vector fields $X$;
\item $Carr(X)$ the correction associated to a vector fields $X$;
\item $Carr^\bullet$ the mould of the correction;
\item $n$ element of $A(X)$;
\item $\textbf{n}$ element of $A^*(X)$;
\item $B_n$ element of $\mathcal{B}(X)$;
\item $B_{\textbf{n}}$ element of $(\mathcal{B}(X), \circ)$;
\item $\omega$ the weight application;
\item $p$ the depth application;
\item $\textit{ping}$ the application which inverts the two components of a letter;
\item $\textit{ret}$ the application on word which inverts the order of the letters;
\end{itemize}

\section{Properties of the Correction's mould}

To prove the different results about the correction, we use another equivalent definition of the mould of the correction using a prenormal form (see \cite{ev2}, p.267 \textbf{Lemma 3.2}) :
\begin{align*}
I^{\textbf{n}} - Carr^{\textbf{n}} = \underset{n \rightarrow + \infty }{\lim}\left((I^\bullet - M^\bullet)^{\circ \ r}\right)^{\textbf{n}}
\end{align*}
where $\textbf{n}=n_1\cdot... \cdot n_r$, $I^\bullet$ is the unit mould for composition, that is $I^{\textbf{n}}=1$ if $\ell (\textbf{n})=1$ and $I^{\textbf{n}}=0$ for $\ell (\textbf{n}) \neq 1$. The mould $M^\bullet$ is a mould of a prenormal form (see \cite{cr1}), so we have $M^{\textbf{n}}=0$ for any non resonant word $\textbf{n}$ and for the empty word. Moreover we have :
\begin{align*}
I^{\textbf{n}}-Carr^{\textbf{n}}=((I^\bullet - M^\bullet)^{\circ \ r})^{\textbf{n}}=((I^\bullet - M^\bullet)^{\circ \ r+k})^{\textbf{n}}
\end{align*}
where $k \in \NN$.
 We remind the composition of two moulds $M^{\bullet}$ and $N^{\bullet}$ :
 \begin{align*}
 ( M^{\bullet} \circ N^{\bullet})^{\textbf{n}}=\underset{1 \leq k \leq \ell (\textbf{n})}{\sum} \underset{\omega_1 \cdot ... \cdot \omega_k}{\overset{*}{\sum}} M^{\parallel \omega_1 \parallel \cdot ... \cdot \parallel \omega_k \parallel}\times N^{\omega_1}\times... \times N^{\omega_k}
\end{align*}  
where $\underset{\omega_1 \cdot ... \cdot \omega_k}{\overset{*}{\sum}}$ means the sum on all the decomposition of the word $\textbf{n}$ in k words. Moreover, $\parallel w_j \parallel$ is a letter obtained by the word $\omega_j$ summing all its letter if the alphabet is provided with a law of semi-group. For more detail, we can see \cite{cr2}.
Now, we can prove the Lemma 1 and Lemma 2.

\subsection*{Proof of Lemma 1}

$1)$ Using the above definition of the mould we have :
\begin{align*}
I^\emptyset-Carr^{\emptyset}&=((I^\bullet)-M^\bullet)^{\circ 0}) ^{\emptyset}\\
&=I^\emptyset-M^\emptyset,
\end{align*}
as $I^\emptyset=M^\emptyset=0$, we have $Carr^\emptyset=0$.
 \\ 
 \\
$2)$ Let $\textbf{n}=n_1\cdot ... \cdot n_r$ a non resonant word of length $r$, so we have :
\begin{align*}
I^{\textbf{n}}-Carr^{\textbf{n}}=((I^\bullet - M^\bullet)^{\circ \ r})^{\textbf{n}},
\end{align*}
if $\textbf{n}$ is a letter i.e $\ell (\textbf{n})=1$, we have : 
\begin{align*}
I^{\textbf{n}}-Carr^{\textbf{n}}&=((I^\bullet - M^\bullet)^{\circ \ 1})^{\textbf{n}}\\
&=(I^\bullet-M^\bullet)^{\parallel \textbf{n} \parallel }(I^\bullet-M^\bullet)^{\textbf{n}},
\end{align*}
where $M^{\parallel \textbf{n} \parallel}=M^{\textbf{n}}=0$ and $I^\textbf{n}=I^{\parallel \textbf{n} \parallel}=1$ hence $Carr^\textbf{n}=0$. \\ 
If the length of $\textbf{n}$ is greater than $2$ :
\begin{align*}
I^{\textbf{n}}-Carr^{\textbf{n}}&=\left( (I^{\bullet}-M^{\bullet} )^{\circ \ r}\right)^{\textbf{n}}\\
&=\left( (I^{\bullet}-M^{\bullet} )^{\circ \ r+1}\right)^{\textbf{n}} \\
&=\left( (I^{\bullet}-M^{\bullet}) \circ (I^{\bullet} - M^{\bullet})^{\circ \ r} \right)^{\textbf{n}}  \\
&=\left((I^{\bullet} - M^{\bullet}) \circ  (I^{\bullet}-Carr^{\bullet}) \right)^{\textbf{n}} \\
&= \underset{1 \leq k \leq r}{\sum} \ \ \underset{w_1\cdot ... \cdot w_k=\textbf{n}}{\sum}(I^{\bullet}-M^{\bullet})^{\parallel w_1 \parallel \cdot ... \cdot \parallel w_k \parallel}(I^{\bullet}-Carr^{\bullet})^{w_1}...(I^{\bullet}-Carr^{\bullet})^{w_k}.
\end{align*}
By induction, we assume the result is true in length $r-1$. As \textbf{n} is non resonant, among all the decomposition in $k$ sub-word, there is at least one of these sub-words which is non resonant for each sum and we denote by $w_j$ this word. So by induction, $I^{w_j}-Carr^{w_j}=0$ if $\ell (w_j)\geq 2$. \\
 If this sub-word $w_j$ is a letter, we treat with the term $(I^{\bullet} - M^{\bullet})^{\parallel w_1 \parallel \cdot ... \cdot \parallel w_j \parallel \cdot ... \cdot \parallel w_k \parallel}$, as $\textbf{n}$ is non resonant and $\parallel w_1 \parallel \cdot ... \cdot \parallel w_j \parallel \cdot ... \cdot \parallel w_k \parallel$ also is, $M^{\parallel w_1 \parallel \cdot ... \cdot \parallel w_j \parallel \cdot ... \cdot \parallel w_k \parallel}=0$  by definition of a prenormal normal, moreover $\ell(\textbf{n})\geq 2$ , $I^{\parallel w_1 \parallel \cdot ... \cdot \parallel w_j \parallel \cdot ... \cdot \parallel w_k \parallel}=0$.
\\
\\
$3)$ We prove this result by induction on the length.
If $\ell (\textbf{n})=2$, $\textbf{n}=n_1\cdot n_2$. If $n_1$ is resonant, as \textbf{n} is resonant, $n_2$ is also resonant. So we have :
\begin{equation}
\left .
\begin{array}{lll}
Carr^{\textbf{n}} & = & -(I^{\parallel\textbf{n}\parallel}-Carr^{\parallel \textbf{n} \parallel})(I^{\textbf{n}}-Carr^{\textbf{n}})\\
& & -(I^{\parallel n_1 \parallel \cdot \parallel n_2 \parallel}-Carr^{\parallel n_1 \parallel \cdot \parallel n_2 \parallel})(I^{n_1}-Carr^{n_1})(I^{n_2}-Carr^{n_2}).
\end{array}
\right .
\end{equation}
As $I^{\parallel\textbf{n}\parallel}-Carr^{\parallel \textbf{n} \parallel}=1-1=0$ and $I^{n_1}-Carr^{n_1}=0$, we have $Carr^{\textbf{n}}=0$ if $\ell (\textbf{n})=2$. Now we assume the result is verified in length $r-1$. We consider $\textbf{n}=n_1\cdot...\cdot n_r$ such that one letter $n_j$ is resonant.
We have to study the following equality, 
\begin{equation}
\left .
\begin{array}{lll}
I^{\textbf{n}}-Carr^{\textbf{n}}& = & \left[ (I^{\bullet}-M^{\bullet} )^{\circ \ r}\right ]^{\textbf{n}} ,\\
 & = & \underset{1 \leq k \leq r}{\sum} \ \ \underset{w_1\cdot ... \cdot w_k=\textbf{n}}{\sum}(I^{\bullet}-M^{\bullet})^{\parallel w_1 \parallel \cdot ... \cdot \parallel w_k \parallel}(I^{\bullet}-Carr^{\bullet})^{w_1}...(I^{\bullet}-Carr^{\bullet})^{w_k}.
\end{array}
\right .
\end{equation}
There exists an integer $l$ such that $n_j$ appears in the decomposition of one $w_l$ for $1\leq l\leq k$. \\
Either $\ell (w_l)=1$, so $w_l=n_j$ and $I^{w_l}-Carr^{w_l}=1-1=0$, or $\ell (w_l)\geq 2$, so by induction hypothesis $Carr^{w_l}=0$ and $I^{w_l}=0$ by definition so $Carr^{\textbf{n}}=0$.

\subsection*{Proof of Lemma 2}

$1)$ Let $n$ a letter such that $\omega(n)=0$. By the above definition, 
\begin{align*}
I^n-Carr^n=I^n-M^n,
\end{align*}
so $Carr^n=M^n$, where $M^{\bullet}$ is a prenormal form. We can take for example the mould $Tram^{\bullet}$ (see \cite{cr2}), which is the mould of the Poincar\'e-Dulac normal form. So $Carr^n=Tram^n=1$. 
\\
\\
$2)$ If $\omega(n_1n_2)=0$, we have $\omega(n_1)=-\omega(n_2)$ as $\omega$ is an morphism. Using the $\textbf{\textit{Theorem 5}}$, we have :
\begin{align*}
\omega(n_1)Carr^{n_1\cdot n_2}+Carr^{n_1+n_2}&=Carr^{n_1}Carr^{n_2}+Carr^{n_1n_2}Carr^{\emptyset}\\
&=0
\end{align*}
because $Carr^{n_1}=Carr^{n_2}=Carr^{\emptyset}=0$ as $\omega(n_i)\neq 0$. Finally : 
\begin{align*}
Carr^{n_1\cdot n_2}=\frac{-1}{\omega(n_1)}.
\end{align*}
We also can prove this result by the other definition :
\begin{align*}
I^{n_1\cdot n_2}-Carr^{n_1\cdot n_2}&=\left( (I^{\bullet}-M^{\bullet})^{\circ 2}\right)^{n_1\cdot n_2}\\
&=\left( (I^{\bullet}-M^{\bullet}) \circ (I^{\bullet}-Carr^{\bullet}) \right)^{n_1 \cdot n_2}\\
&=(I^{\bullet}-M^{\bullet})^{\parallel n_1 \cdot n_2 \parallel}(I^{\bullet}-Carr^{\bullet})^{n_1 \cdot n_2}+\\
& +(I^{\bullet}-M^{\bullet})^{\parallel n_1 \parallel \cdot \parallel n_2 \parallel}(I^{\bullet}-Carr^{\bullet})^{n_1}(I^{\bullet}-Carr^{\bullet})^{n_2}.
\end{align*}
As above we will use the mould $Tram^{\bullet}$ in length 2, $Carr^{n_1 \cdot n_2}=Tram^{n_1 \cdot n_2}=\frac{\omega(n_2)-\omega(n_1)}{\omega(n_1)\omega(n_2)}$. As $\omega(n_1)=-\omega(n_2)$, we finally have : 
\begin{align*}
Carr^{n_1 \cdot n_2}=\frac{1}{\omega(n_1)}.
\end{align*}
As we can see the result differs by a multiplication by $-1$. It is due to the fact in \cite{ev1} and \cite{ev2}, the nested Lie brackets are taken in this form $[B_{n_1 \cdot ... \cdot n_r}]=[B_{n_r},[B_{n_{r-1}},[...[B_{n_2},B_{n_1}]..]]$ whereas in \cite{cr1} and \cite{cr2} we consider $[B_{n_1 \cdot ... \cdot n_r}]=[...[B_{n_1},B_{n_2}],...],B_{n_{r-1}}],B_{n_r}]$. And we have the relation: 
\begin{align*}
[B_{n_r},[B_{n_{r-1}},[...[B_{n_2},B_{n_1}]..]]=(-1)^{r+1}[B_{n_1},B_{n_2}],...],B_{n_{r-1}}],B_{n_r}].
\end{align*}
$3)$ We have to remark in length 3, there is not any sub-word which resonant else there is a resonant letter but using the previous lemma the correction is equal to zero. 

Using \textbf{\textit{Theorem 5}} we have:
\begin{equation}
\left .
\begin{array}{lll}
\omega(n_1)Carr^{n_1 \cdot n_2 \cdot n_3}+Carr^{(n_1+n_2) \cdot n_3}&=&
Carr^{n_1 \cdot n_3}Carr^{n_2}\\
& & + Carr^{n_1 \cdot n_2 \cdot n_3}Carr^{\emptyset}+Carr^{n_1}Carr^{n_2\cdot n_3}.
\end{array}
\right .
\end{equation}
We assume $\omega(n_1) \neq 0 \neq \omega(n_2 \cdot n_3)$ and we have to remark $n_1 \cdot n_3$ is not resonant else $n_2$ is resonant too and by the previous lemma the correction is equal to zero, so we have :
\begin{align*}
Carr^{n_1 \cdot n_2 \cdot n_3}&=\frac{1}{\omega(n_1)(\omega(n_1)+\omega(n_2))}.
\end{align*}
As in length 2, we can use the mould $Tram^{\bullet}$ to compute the correction : 
\begin{align*}
I^{n_1\cdot n_2 \cdot n_3}-Carr^{n_1 \cdot n_2 \cdot n_3}&=\left( (I^{\bullet}-Tram^{\bullet})^{\circ \ 3} \right)^{n_1 \cdot n_2 \cdot n_3}\\
&= \left( (I^{\bullet}-Tram^{\bullet})^{\circ \ 4} \right)^{n_1 \cdot n_2 \cdot n_3}\\
&=(I^{\bullet}-Tram^{\bullet}) \circ (I^{\bullet}-Carr^{\bullet})^{n_1 \cdot n_2 \cdot n_3} \\
&= -Tram^{n_1 \cdot n_2 \cdot n_3}.
\end{align*}
So $Carr^{n_1 \cdot n_2 \cdot n_3}=Tram^{n_1 \cdot n_2 \cdot n_3}=\frac{1}{\omega(n_1)(\omega(n_1)+\omega(n_2))}$.

\section{Technical results}

\subsection{Proof of Theorem \ref{rational}}

For all $i=1,\dots ,2p$, we have using the Theorem of projection
\begin{equation}
\left .
\begin{array}{lll}
Carr_{2p,i} (X) & = & \di\frac{1}{i!} \underset{p(\bn )=2p,\, l(\bn )=i}{\underset{\bn\in A^*(X)}{\sum}} Carr^{\bn}[B_{\bn}] ,\\
 & = & \di\frac{1}{i!} \underset{p(\bn )=2p,\, l(\bn )=i}{\underset{\bn\in A^*(X)}{\sum}} Carr^{\bn}
 (xy)^{\frac{p(\textbf{n})}{2}}(P(\textbf{n})x\partial_x+Q(\textbf{n})y\partial_y),\\
 & = & \di\frac{1}{i!} (xy)^{p} \underset{p(\bn )=2p,\, l(\bn )=i}{\underset{\bn\in A^*(X)}{\sum}} Carr^{\bn}
 (P(\textbf{n})x\partial_x+Q(\textbf{n})y\partial_y) .
\end{array}
\right .
\end{equation}

By definition of $Ca_{2p,i} (X)$ we obtain
\begin{equation}
Carr_{2p,i} (X) = \di\frac{1}{i!} (xy)^{p} 
\left [ 
Ca_{2p,i} (X) x\partial_x +\overline{Ca_{2p,i} (X)} y \partial y 
\right ] .
\end{equation}

\subsection{Proof of Lemma \ref{polycro}}

Let $B_n=B_{(n^1,n^2)}=x^{n^1}y^{n^2}(p_n x \partial_x +q_n y \partial_y)$ and $B_m=B_{(m^1,m^2)}=x^{m^1}y^{m^2}(p_mx\partial_x+q_my\partial_y)$. Then, the Lie brackets of $B_n$ and $B_m$ is :
\begin{align*}
[B_{nm}]:=[B_n,B_m]& =x^{n^1+m^1}y^{n^2+m^2}(P_{n,m}x \partial_x+Q_{n,m}y \partial_y),\\
 & =x^{\mid nm\mid^1}y^{\mid nm\mid^2}(P_{n,m}x \partial_x+Q_{n,m}y \partial_y)
\end{align*}
where $P_{n,m}$ and $Q_{n,m}$ are polynomials in the coefficients of $B_n$ and $B_m$, precisely : 
\begin{align*}
P_{n,m}&=(m^1-n^1)p_np_m+m^2q_np_m-n^2p_nq_m, \\ 
Q_{n,m}&=(m^2-n^2)q_nq_m+m^1p_nq_m-n^1q_np_m.
\end{align*}
We easily prove by induction that all the Lie bracket in any length are on the above form.

\end{appendix}

\end{document}